\documentclass[times]{rncauth}

\usepackage{moreverb}

\usepackage[colorlinks,bookmarksopen,bookmarksnumbered,citecolor=red,urlcolor=red]{hyperref}

\newcommand\BibTeX{{\rmfamily B\kern-.05em \textsc{i\kern-.025em b}\kern-.08em
T\kern-.1667em\lower.7ex\hbox{E}\kern-.125emX}}

\def\volumeyear{2010}

\usepackage{graphics} 
\usepackage{amsmath} 
\usepackage{amssymb}  
\usepackage[noadjust,sort]{cite}
\usepackage{color}
\usepackage{array,arydshln}

\DeclareSymbolFont{bbold}{U}{bbold}{m}{n}
\DeclareSymbolFontAlphabet{\mathbbold}{bbold}

\newcommand{\onev}{\mathbbold{1}}

\newtheorem{theorem}{Theorem}[section]
\newtheorem{corollary}[theorem]{Corollary}
\newtheorem{proposition}[theorem]{Proposition}
\newtheorem{lemma}[theorem]{Lemma}
\newtheorem{definition}[theorem]{Definition}

\newtheorem{example}[theorem]{Example}
\newtheorem{problem}[theorem]{Problem}
\newtheorem{remark}[theorem]{Remark}

\usepackage{mathtools}
\mathtoolsset{centercolon}

\newcommand{\norm}[1]{\lVert #1 \rVert}

\newcommand{\oneto}[1]{[#1]}

\DeclareMathOperator*{\minimize}{minimize}
\DeclareMathOperator*{\subjectto}{subject\ to}
\DeclareMathOperator{\tr}{tr}
\let\Pr\undefined\DeclareMathOperator{\Pr}{Pr}

\begin{document}

\runningheads{M.~Ogura, A.~Cetinkaya, T.~Hayakawa, and V.~M.~Preciado}{State Feedback Control with Hidden-Markov Mode Observation}

\title{State feedback control of Markov jump linear systems\\with
hidden-Markov mode observation}

\author{Masaki~Ogura\affil{1}\corrauth
, Ahmet~Cetinkaya\affil{2}, Tomohisa~Hayakawa\affil{2}, and  Victor~M.~Preciado\affil{1}
}

\address{\affilnum{1}Department of Electrical and Systems Engineering, 
University of Pennsylvania, Philadelphia, PA, 19104, USA.\break
\affilnum{2}Department of Mechanical
and Environmental Informatics, Tokyo Institute of Technology,
Tokyo 152-8552, Japan. 
}

\corraddr{Department of Electrical and Systems Engineering, 
University of Pennsylvania, Philadelphia, PA, 19104, USA. Email:  ogura@seas.upenn.edu}

\cgsn{NSF}{CNS-1302222, IIS-1447470.}

\begin{abstract}
In this paper, we study state-feedback control of Markov jump linear systems with partial information. In particular, we assume that the controller can only access the mode signals according to a hidden-Markov observation process. Our formulation generalizes various relevant cases previously studied in the literature on Markov jump linear systems, such as the cases with perfect information, no information, and cluster observations of the mode signals. In this context, we propose a Linear Matrix Inequalities (LMI) formulation to design feedback control laws for (stochastic) stabilization, $H_2$, and $H_\infty$ control of discrete-time Markov jump linear systems under hidden-Markovian observations of the mode signals. We conclude by illustrating our results with some numerical examples.
\end{abstract}


\maketitle

\section{Introduction}

Markov jump linear systems \cite{Costa2005} are an important class of switched
systems in which the \emph{mode signal}, responsible for controlling the switch
among dynamic modes, is modeled by a time-homogeneous Markov process. This type
of systems has been widely used in multiple applications, such as
robotics~\cite{Siqueira2004,Vargas2013a}, economics~\cite{Blair1975}, networked
control~\cite{Hespanha2007a}, and epidemiology~\cite{Ogura2014i}. Solutions to
standard  controller synthesis problems  for Markov jump linear systems, such as
state-feedback stabilization, quadratic optimal control, $H_2$~optimal control,
and $H_\infty$~optimal control~(see, e.g., the monograph~\cite{Costa2005}), can
be found in the literature. These works, however, are based on the unrealistic
assumption that the controller has full knowledge about the mode signal at any
time instant.

To overcome this limitation, several papers investigate the effect of limited
and/or uncertain knowledge about the mode signal. For example, the authors
in~\cite{DoVal2002} studied $H_2$~control of discrete-time Markov jump linear
systems when the state space of the mode signal is partitioned into subsets,
called \emph{clusters}, and the controller only knows in which cluster the mode
signal is at a given time. Similar studies in the context of $H_\infty$~control
can be found in \cite{Goncalves2012,Fioravanti2014}. In the extreme case of
having a single mode cluster (in other words, when one cannot observe the mode),
the authors in~\cite{Vargas2013,Vargas2015} investigated quadratic optimal
control problems. Most of the above works can be studied in a framework based on
random and uncertain mode observations (see \cite{Costa2014b} for the
description of this framework in the context of $H_2$~control). In a
complementary line of work, we find some papers assuming that the mode signal
can only be observed at particular sampling times, instead of at any time
instant. In this direction, we find in the literature a variety of random
sampling strategies of the mode signal. The authors in~\cite{Cetinkaya2014b}
designed almost-surely stabilizing state-feedback gains when the sampling times
follow a renewal process. Similarly, the authors
in~\cite{Cetinkaya2013a,Cetinkaya2014} derived stabilizing state-feedback gains
using Lyapunov-like functions under periodic observations.

In this paper, we propose a framework to design state-feedback controllers for
discrete-time Markov jump linear systems assuming that the mode signal can only
be observed when a Markov chain (different than the one describing the mode
signal) visits a particular subset of its state space. We call this observation
process \emph{hidden-Markov}, due to its similitude with hidden-Markov
processes~\cite{Ephraim2002}. We show how hidden-Markov observation processes
generalize many relevant cases previously studied in the literature, such as
those
in~\cite{DoVal2002,Goncalves2012,Cetinkaya2013a,Cetinkaya2014,Cetinkaya2014b,Costa2014b}. In this context, we propose a Linear Matrix Inequalities (LMI) formulation to design feedback control laws for (stochastic) stabilization, $H_2$, and $H_\infty$ control of discrete-time Markov jump linear systems under hidden-Markov observations of the mode signal. It is important to remark that, since the observation process is hidden-Markovian, existing control synthesis methods for Markov jump linear systems, such as those in~\cite{Costa2005,DoVal2002,Goncalves2012}, do not apply to our case.

The paper is organized as follows. In Section~\ref{sec:PrbFormulation}, we
formulate the state-feedback control problem for Markov jump linear systems with
hidden\nobreakdash-Markovian observations of the mode signal. We show in
Section~\ref{sec:analysis} that the resulting closed-loop system can be reduced
to a standard Markov jump linear system by embedding the (possibly
non-Markovian) stochastic processes relevant to the controller into an extended
Markov chain. In Section~\ref{sec:design}, we derive an LMI formulation to
design state-feedback gains for stabilization, $H_2$, and $H_\infty$ control
problems. Finally, in Section~\ref{sec:example}, we illustrate our results with
some numerical examples.

\subsection*{Notation}
The notation used in this paper is standard. Let $\mathbb{N}$ denote the
set of nonnegative integers. Let $\mathbb{R}^n$ and $\mathbb{R}^{n\times
m}$ denote the vector spaces of real $n$-vectors and $n\times m$ matrices,
respectively. By $\norm{\cdot}$, we denote the Euclidean norm
on~$\mathbb{R}^n$. $\Pr(\cdot)$ will be used to denote the probability of
an event. The probability of an event conditional on another event
$\mathcal A$ is denoted by $\Pr(\cdot \mid \mathcal A)$. Expectations are
denoted by $E[\cdot]$. For a positive integer $N$, we define the set $[N] =
\{1, \dotsc, N\}$. For a positive integer~$T$ and an integer $k$, define
$\lfloor k \rfloor_T$ as the unique integer in $\{0, \dotsc, T-1\}$ such
that $k - \lfloor k\rfloor_T$ is an integer multiple of $T$. When a real
symmetric matrix $A$ is positive (resp., negative) definite, we write $A>0$
(resp., $A<0$). The notations $A\geq 0$ and $A\leq 0$ are then understood
in the obvious way. For sets of matrices $A = \{A_\lambda\}_{\lambda \in
\Lambda} \subset \mathbb{R}^{n\times m}$ and $B = \{B_\lambda\}_{\lambda
\in \Lambda} \subset \mathbb{R}^{m\times \ell}$ sharing the same index set
$\Lambda$, we define another set of matrices $AB = \{A_\lambda
B_\lambda\}_{\lambda \in \Lambda} \subset \mathbb{R}^{n\times \ell}$. The
symbol~$\star$ will be used to denote the symmetric blocks of partitioned
symmetric matrices. Finally, indicator functions are denoted by
$\onev(\cdot)$.

\section{Problem formulation}\label{sec:PrbFormulation}

In this section, we formulate the problems under study. Let
$n$, $m$, $q$, $\ell$, and~$N$ be positive integers. For each $i\in [N]$, let
$A_i \in \mathbb{R}^{n\times n}$, $B_i\in \mathbb{R}^{m\times n}$, $C_i \in
\mathbb{R}^{\ell \times n}$, $D_i\in \mathbb{R}^{\ell \times m}$, and
\mbox{$E_i\in \mathbb{R}^{n\times q}$}.  Also, let $r = \{r(k) \}_{k=0}^\infty$
be the time-homogeneous Markov chain taking values in~$[N]$ and having the
transition probability matrix~$P\in \mathbb{R}^{N\times N}$. Consider the
Markov jump linear system~\cite{Costa2005}:
\begin{equation*}
\Sigma:\begin{cases}
\begin{aligned}
x(k+1) &= A_{r(k)} x(k) + B_{r(k)}u(k) + E_{r(k)}w(k),
\\
z(k)   &= C_{r(k)}x(k) + D_{r(k)}u(k).
\end{aligned}
\end{cases}
\end{equation*}
We call $x$ and $r$ the state and the mode of $\Sigma$, respectively. The
signal $w$ represents an exogenous disturbance, $u$ is the control input,
and $z$ is the measured signal. The initial conditions are denoted by $x(0) = x_0$
and $r(0) = r_0$. We will assume that $x_0$ and $r_0$ are either deterministic
constants or random variables, depending on the particular control problems considered.

\subsection{State-feedback control with hidden-mode observation}

In this paper, we consider the situation where the controller cannot measure the
mode signal at every time instant. To study this case, we model the times at
which the controller can observe the mode by the stochastic process $t =
\{t_i\}_{i=0}^\infty$ taking values in $\mathbb{N}\cup \{\infty\}$. We call $t$
the \emph{observation process} and each $t_i$ an \emph{observation time}. For
each $i$, we assume either $t_i<t_{i+1}$ or $t_i = t_{i+1} = \infty$. It is understood that, if $t_i < t_{i+1} = \infty$, then no observation
will be performed after time~$t_i$.

In this paper, we focus on the following class of observation processes:

\begin{definition} \label{Def:HiddenMarkov}
We say that an observation process $t$ is \emph{hidden-Markov}\footnote{We
adopt the terminology ``hidden-Markov'' because the process
$\{f(s(k))\}_{k\geq 0}$ characterizing the observation process~$t$ is a
hidden\nobreakdash-Markov process~\cite{Ephraim2002}.} if there exist an $M\in
\mathbb{N}$, a Markov chain \mbox{$s = \{s(k)\}_{k\geq 0}$} taking values
in $[M]$ (independent of the mode $r$), and a function~\mbox{$f \colon [M] \to
\{0, 1\}$} such that
\begin{equation*} 
t_0 = \min \{k\geq 0 : f(s(k)) = 1\}
\end{equation*}
and, for every $i\geq 0$, 
\begin{equation*} 
t_{i+1} = \min\{k > t_i : f(s(k)) = 1\}, 
\end{equation*}
where the minimum of the empty set is understood to be $\infty$. 
\end{definition}

For example, if the image of $f$ equals the set $\{1\}$, then the controller
observes the mode at all time instants. On the other hand, if $f$ maps into
$\{0\}$, then the controller never observes the mode signal. In fact, the class
of hidden-Markov observation processes contains many other interesting examples
as will be seen below. Throughout the paper, we denote the transition
probability matrix of $s$ by~$Q\in\mathbb{R}^{M\times M}$. In what follows, we
provide three particular examples that can be formulated as hidden-Markovian
observation processes:

\begin{example}[Gilbert-Elliot channel]\label{ex:GEchannel}
Consider the case where the controller observes the mode through a
Gilbert-Elliot channel~\cite{Gilbert1960}. This channel has two possible states:
the good (G) and bad (B) states. When the channel is at the G state, it
transmits the mode signal to the controller; in contrast, when it is at state B,
it does not transmit. This channel switches its state according to a Markov
chain, defined as follows. Let $p, q\in [0, 1]$ be the transition probabilities
from G to B and B to G, respectively. We can formulate this channel as a
hidden-Markovian observation process (Definition \ref{Def:HiddenMarkov}) using
the following parameters:
\begin{equation*}
M=2,\quad 
Q = \begin{bmatrix}
1 - p & p\\
q & 1 - q
\end{bmatrix},\quad
f(\lambda)=
\begin{cases}
1, & \text{if $\lambda = 1$}, 
\\
0, & \text{if $\lambda = 2$}.
\end{cases}
\end{equation*}
\end{example}

Our second example is closely related to the observation processes investigated in~\cite{Costa2014b}:

\begin{example}[{Observations with independent and identically distributed failures}]\label{ex:IID}
Assume that, at each time instant, the controller attempts to observe the mode
signal but it fails with probability $p_f \in [0, 1]$, independently from the
observations at other time instants. This observation process can be implemented
as a hidden-Markovian observation process using the following parameters:
\begin{equation*}
M=2,\quad 
Q = \begin{bmatrix}
1-p_f & p_f\\1-p_f & p_f
\end{bmatrix},\quad
f(\lambda)=
\begin{cases}
1, & \text{if $\lambda = 1$}, 
\\
0, & \text{if $\lambda = 2$}.
\end{cases}
\end{equation*}
We remark that, under a similar problem setting, the authors in
\cite{Costa2014b} propose a framework for stochastic stabilization and $H_2$
control of Markov jump linear systems.
\end{example}

Our last example is concerned with periodic observation with failures: 

\begin{example}[Periodic observation with failures] 
Let $\ell$ be a positive integer and $p\in [0, 1]$. Define
\begin{equation*}
M = \ell+1,\quad 
Q = \left[\begin{array}{cc:ccc}
&&1\\\hdashline
&&1\\
&&&\ddots\\
&&&&1\\ \hdashline
p&1-p
\end{array}
\right]
,\quad
f(\lambda)=
\begin{cases}
1, & \text{if $\lambda = 1$}, 
\\
0, & \text{otherwise.}
\end{cases}
\end{equation*}
Then, we can see that $t_{i+1} - t_i$ is a positive integer multiple of $\ell$
with probability one and, also, $\Pr(t_{i+1} - t_i = k\ell) = (1-p)^{k-1}p$ for
all $i\geq 0$ and $k\geq 1$. The corresponding observation process describes the
situation where the controller tries to observe the mode signal every
$\ell$ time units with a probability of success $p$ for each observation.
In particular, for $p=1$, this observation process gives the periodic
case considered in~\cite{Cetinkaya2013a,Cetinkaya2014}.
\end{example}

In order to specify the behavior of the controller between two consecutive
observation times, we introduce the following processes. Given an observation
process~$t$, we define the stochastic process~$\tau = \{\tau(k)\}_{k=0}^\infty$
by
\begin{equation*}
\tau(k) = 
\begin{cases}
\max\{ t_i: t_i \leq k,\, i\geq 0\},  & \text{if }k \geq t_0, 
\\
\tau_0, &\text{otherwise,}
\end{cases}
\end{equation*}
where $\tau_0$ is an integer satisfying
\begin{equation}\label{eq:def:tau_0}
\begin{cases}
\tau_0 = 0, &\text{if $t_0 = 0$}, 
\\
\tau_0 < 0, &\text{otherwise}.
\end{cases}
\end{equation}
For each time $k$, the above defined $\tau(k)$ represents the most recent time
the controller observed the mode. We, in particular, have $\tau(t_i) = t_i$ for
every $i\geq 0$. Notice that, for $k<t_0$, we augment the process $\tau$ with a
negative integer $\tau_0$. This is because, before time $k = t_0$, no
observation is performed by the controller yet. This augmentation is not needed
if $t_0 = 0$, in which case we set $\tau_0 = 0$ as in~\eqref{eq:def:tau_0}.

We also define the stochastic process~$\sigma =
\{\sigma(k)\}_{k=0}^\infty$ taking values in $[N]$ by
\begin{equation*}
\sigma(k) = 
\begin{cases}
r(\tau(k)), & \text{if }k\geq t_0, 
\\
\sigma_0, &\text{otherwise,}
\end{cases}
\end{equation*}
where $\sigma_0$ is an element in $[N]$ satisfying
\begin{equation}\label{eq:def:sigma0}
[t_0 = 0] \Rightarrow [\sigma_0 = r_0].
\end{equation}
\begin{figure}[tb]
\vspace{.1in} \centering \includegraphics[width=7cm]{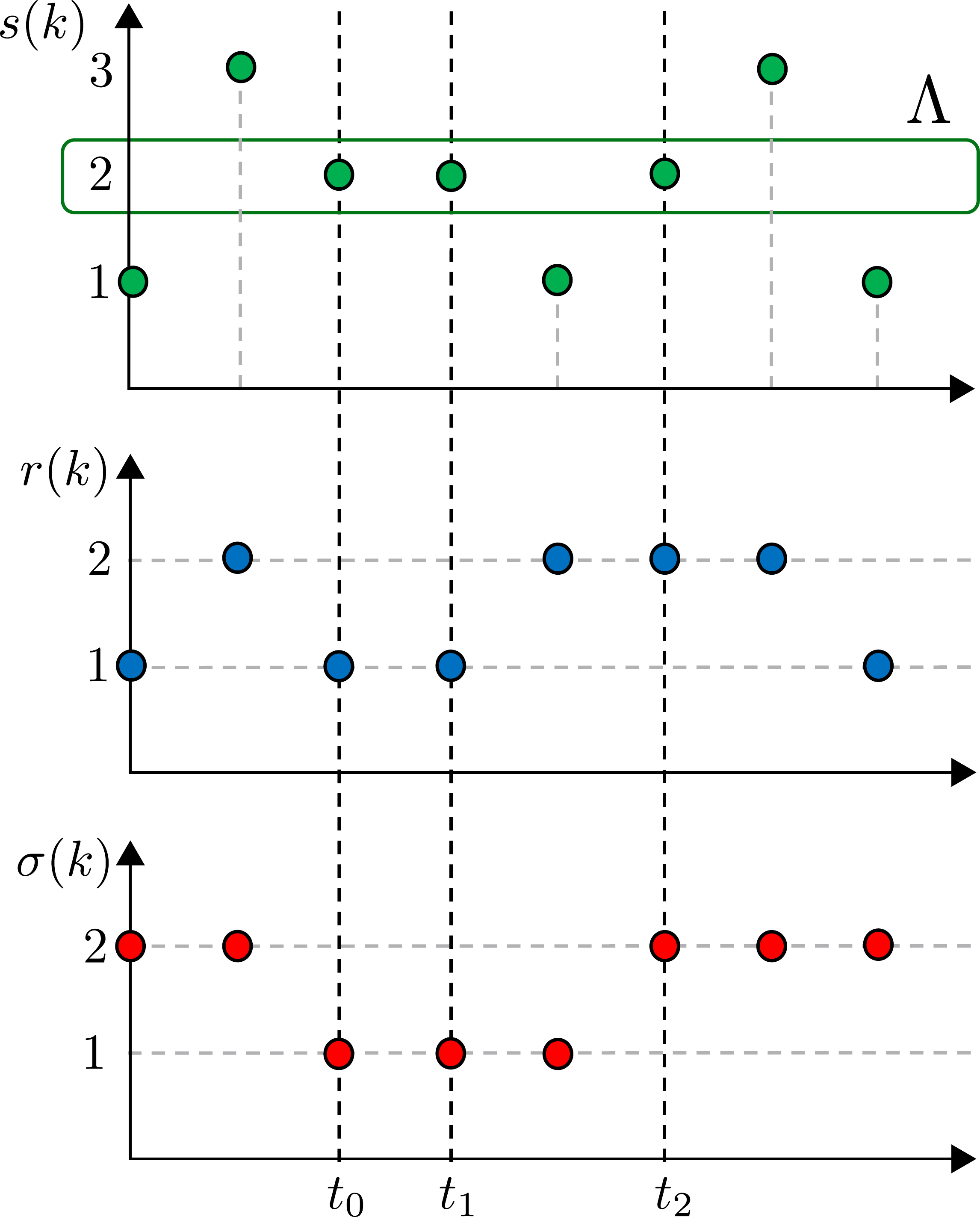}
\caption{An observation of the mode~$r$. The observation times $t_0$, $t_1$, $t_2$, $\dotsc$ are determined by the Markov chain $s$ and the function $f\colon [3] \to \{0, 1\}$ given by $f(s)=1$ if $s \in \Lambda =\{2\}$ and $f(s)=0$ otherwise. Until the first observation time~$t_0 = 2$, the most recent observation~$\sigma$
is temporarily set to $\sigma_0 = 2$.} \label{fig:observation}
\end{figure}%
For each $k$, the random variable $\sigma(k)$ represents the most-updated
information about the mode signal kept by the controller at time $k$. We again
notice that, by the same reason indicated above, $\sigma$ is augmented by an
arbitrary $\sigma_0$ before the time instant $k=t_0$ (i.e., before the first
observation is performed). As is the case for $\tau_0$, if $t_0 = 0$, then this
augmentation is not needed and thus we set $\sigma_0 = r_0$ as in
\eqref{eq:def:sigma0}. See Figure~\ref{fig:observation} for an illustration of
the stochastic processes described so far.

In what follows, we present the state-feedback control scheme studied in this
paper. We assume that the controller has an access to the following pieces of
information at each time $k\geq 0$: (\emph{i}) the state variable~$x(k)$,
(\emph{ii}) the most recent observation~$\sigma(k)$ of the mode~$r$, and
(\emph{iii})~the quantity~{$k-\tau(k)$}, which is the time elapsed since the
last observation. Specifically, the state-feedback controller under
consideration takes the form
\begin{equation}\label{eq:state-fb}
u(k) = K_{\sigma(k), \lfloor k-\tau(k) \rfloor_T + 1} x(k),
\end{equation}
where $K_{\gamma, \delta} \in \mathbb{R}^{m\times n}$ for each $\gamma \in [N]$
and $\delta \in \oneto T$. The first subindex of $K$ in \eqref{eq:state-fb},
i.e., $\sigma(k)$, allows the gain to be reset whenever the controller observes
the mode. The second subindex\footnote{Notice that we add $1$ to $\lfloor k -
\tau(k) \rfloor_T$ in the second subindex of $K$ in \eqref{eq:state-fb} to make
the index~$\delta$ of $K_{\gamma, \delta}$ start from $1$, instead of~$0$.} of
$K$ in \eqref{eq:state-fb} allows the controller to change its feedback gain
between two consecutive observation times with period $T$ as in
\cite{Cetinkaya2014}, rather than keeping them to be constant. We will later see
in Section~\ref{sec:example} that, as the period $T$ increases, the performance
of the controller can in fact improve. Throughout the paper, we will use the
notation
\begin{equation*}
\rho(k) = \lfloor k-\tau(k) \rfloor_T + 1. 
\end{equation*}
Notice that the initial condition of $\rho$ is given by
$\rho(0) = \rho_0 = \lfloor-\tau_0 \rfloor_T + 1$.

\subsection{Performance measures}

We now introduce several performance measures used to evaluate the
state-feedback control law~\eqref{eq:state-fb}. The feedback control law
\eqref{eq:state-fb} applied to $\Sigma$ yields the following closed-loop system
\begin{equation*}
\Sigma_K : \begin{cases}
\begin{aligned}
x(k+1) &= \left(A_{r(k)} + B_{r(k)} K_{\sigma(k), \rho(k)}\right) x(k) + E_{r(k)}w(k),
\\
z(k)   &= \left(C_{r(k)} + D_{r(k)} K_{\sigma(k), \rho(k)}\right)x(k).
\end{aligned}
\end{cases}
\end{equation*}
Let us introduce the compact notation
\begin{equation}\label{eq:def:bar r}
\bar r(k) = (r(k), s(k), \sigma(k), \rho(k)). 
\end{equation}
Also, define $\mathfrak X$ as the set of quadruples $(\alpha, \beta, \gamma,
\delta) \in [N]\times [M]\times [N]\times [T]$ such that, if $ f(\beta) = 1$,
then $\alpha = \gamma$ and $\delta = 1$. The set $\mathfrak X$ contains all
possible values that can be taken by the stochastic process $\bar r$. We denote
the initial condition for $\bar r$ as $\bar r(0) = \bar r_0$. We sometimes
denote the trajectories~$x$ and $z$ of $\Sigma_K$ by $x(\cdot; x_0, \bar r_0,
w)$ and $z(\cdot; x_0, \bar r_0, w)$, respectively, whenever we need to clarify
the initial conditions as well as the disturbance~$w$. We finally remark that,
by the conditions in \eqref{eq:def:tau_0} and \eqref{eq:def:sigma0}, $\bar r_0$
is determined by $r_0$, $s_0$, $\sigma_0$, and $\rho_0$ as
\begin{equation}\label{eq:def:barr_0}
\bar r_0 = \begin{cases}
(r_0, s_0, r_0, 1), &\text{if $f(s_0) = 1$}, 
\\
(r_0, s_0, \sigma_0, \rho_0), &\text{otherwise}.
\end{cases}
\end{equation}

The first performance measure under consideration is mean square stability:

\begin{definition}[Mean square stability]\label{defn:ms}
We say that $\Sigma_K$ is \emph{mean square stable} if there exist $C>0$ and
$\lambda \in [0, 1)$ such that $E[\norm{x(k)}^2] \leq C\lambda^k \norm{x_0}^2$
for all $x_0$, $\bar r_0$, and $k\geq 0$, provided $w \equiv 0$.
\end{definition}

In order to define the second performance under consideration, we first need to
introduce the space of square summable stochastic processes, as follows. Let
$\Theta_k$ be the $\sigma$-algebra generated by the random variables $\{r(k),
s(k), \dotsc, r(0), s(0)\}$. Define $\ell^2(\mathbb{R}^n)$ ($\ell^2$ for short)
as the space of $\mathbb{R}^n$\nobreakdash-valued stochastic processes $f =
\{f(k)\}_{k\geq 0}$ such that $f(k)$ is an $\mathbb{R}^n$-valued and
$\Theta_k$-measurable random variable each $k\geq 0$ and, moreover,
$\sum_{k=0}^\infty E[\norm{f(k)}^2]$ is finite. For $f\in\ell^2$, define its
$\ell^2$-norm $\norm{f}_2$ by $\norm{f}_2^2 = \sum_{k=0}^\infty
E[\norm{f(k)}^2]$. Then, we extend the definition of the $H_2$~norm of a Markov
jump linear system introduced in~\cite{DoVal2002}, as follows:

\begin{definition}[$H_2$ norm]
Assume that $\bar r_0$  follows the probability distribution $\bar \mu$. Define
the \emph{$H_2$~norm} of $\Sigma_K$ by
\begin{equation*}
\norm{\Sigma_K}_2 = \sqrt{\sum_{i=1}^q \sum_{\chi \in \mathfrak X} 
\bar \mu(\chi)
\norm{z(\cdot; 0, \chi, e_i\phi)}_2^2}, 
\end{equation*}
where $e_i$ denotes the $i$th standard unit vector in $\mathbb{R}^q$ and $\phi$
is the function defined on $\mathbb{N}$ by $\phi(0) = 1$ and $\phi(k) = 0$ for
$k\geq 1$.
\end{definition}

Our third and last performance measure is the $H_\infty$~norm. In our context,
we use the following definition, which is an extension of the one for standard
Markov jump linear systems~\cite{Seiler2003}:

\begin{definition}[$H_\infty$~norm]
Assume that $\Sigma_K$ is mean square stable and $x_0 = 0$. Define the
\emph{$H_\infty$~norm} of $\Sigma_K$ by
\begin{equation*}
\norm{\Sigma_K}_\infty = \sup_{\bar r_0 \in \mathfrak X} \sup_{w\in \ell^2(\mathbb{R}^q)\backslash \{0\}} 
\frac{\norm{z}_2}{\norm{w}_2}.
\end{equation*}
\end{definition}

Having introduced these three performance measures, we can now formulate the
problems under consideration. The first problem is concerned with stochastic
stabilization, which is stated as follows:

\begin{problem}[Stabilization]\label{prb:stbl}
Find a set of matrices $\{K_{\gamma, \delta}\}_{\gamma\in[N],
\delta\in[T]}\subset \mathbb{R}^{m\times n}$ such that $\Sigma_K$ is mean
square stable.
\end{problem}

The second problem is concerned with the stabilization of the closed-loop system
with an upper-bound on the $H_2$ norm. In this problem, we assume that the
distributions of $r_0$ and $s_0$ are given. Thus, the parameters to be designed
are feedback gains $K$ and the distribution $\nu$ of the pair $(\sigma_0,
\rho_0)$:

\begin{problem}[$H_2$ control]\label{prb:H2}
Assume that the distributions of $r_0$ and $s_0$ are known. For a given $\gamma
> 0$, find a set of matrices $\{K_{\gamma, \delta}\}_{\gamma\in[N],
\delta\in[T]}\subset \mathbb{R}^{m\times n}$ and a distribution $\nu$ such that
$\Sigma_K$ is mean square stable and $\norm{\Sigma_K}_2 < \gamma$.
\end{problem}

\begin{remark}
Given the distribution $\mu_r$ (resp., $\mu_s$) of $r_0$ (resp., $s_0$),
using  \eqref{eq:def:barr_0} we can find $\bar \mu$ as 
\begin{equation}\label{eq:nu.lin}
\bar \mu(\chi)
=
\begin{cases}
\mu_r(\alpha)\mu_s(\beta) &  \text{if $f(\beta) = 1$}, 
\\
\mu_r(\alpha)\mu_s(\beta)\nu(\gamma, \delta) &\text{otherwise}, 
\end{cases}
\end{equation}
 for every $\chi =
(\alpha, \beta, \gamma, \delta) \in \mathfrak X$.
\end{remark}

The last problem is the stabilization of the closed-loop system with an
upper-bound on the $H_\infty$ norm:

\begin{problem}[$H_\infty$ control]\label{prb:Hinfty}
For a given $\gamma > 0$, find a set of matrices $\{K_{\gamma,
\delta}\}_{\gamma\in[N], \delta\in[T]}\subset \mathbb{R}^{m\times n}$ such that
$\Sigma_K$ is mean square stable and $\norm{\Sigma_K}_\infty < \gamma$.
\end{problem}

We remark that $\Sigma_K$ is no longer a standard Markov jump linear system, due
to the nature of the processes $\sigma$ and $\rho$. Therefore, we cannot use any
of the techniques in the literature~\cite{Costa2005,DoVal2002,Goncalves2012} to
synthesize a control law. In fact, we cannot use existing techniques even to
analyze the performance of $\Sigma_K$. Moreover, due to the generality of
hidden-Markov observation processes, we cannot use any of the results in
\cite{Cetinkaya2014b,Cetinkaya2014,Costa2014b}, recently proposed to design
state-feedback control laws for Markov jump linear systems with partial mode
observation.

\section{Analysis of closed-loop system}\label{sec:analysis}

In this section, we show how to analyze the  closed-loop system~$\Sigma_K$. In
this direction, we reduce $\Sigma_K$ to a standard Markov jump linear system by
embedding stochastic processes appearing in the closed-loop system (which are
not necessarily Markovian) into an extended Markov chain with a larger state
space. Let us begin with the following observation:

\begin{lemma}\label{lem:extendedMarkov}
The stochastic process $\bar r$ defined by \eqref{eq:def:bar r} is a
time-homogeneous Markov chain. Moreover, its transition probabilities are given
by
\begin{equation}\label{eq:bar_p}
\Pr(\bar r(k+1) = \chi' \mid \bar r(k) = \chi) 
= 
\begin{cases}
\onev(\alpha' = \gamma\,',\,\delta'=1) p_{\alpha \alpha'} q_{\beta \beta'}, 
&
\text{if } f(\beta') = 1, 
\\
\onev(\gamma\,' = \gamma,\, \lfloor\delta' - \delta -1 \rfloor_T = 0) p_{\alpha \alpha'} q_{\beta \beta'}, 
&
\text{otherwise,}
\end{cases}
\end{equation}
for all $\chi = (\alpha, \beta, \gamma, \delta)$ and {$\chi' = (\alpha', \beta',
\gamma\,', \delta')$} in $\mathfrak X$.
\end{lemma}

\begin{proof}
Let $k_0\in \mathbb{N}$, $k\geq k_0$, and $\chi_i = (\alpha_i, \beta_i,
\gamma_i, \delta_i)\in \bar{\mathfrak X}$ ($i=k_0, \dotsc, k+1$) be arbitrary.
For each $i$, define the events $\mathcal A_i$ and $\mathcal B_i$ as $\mathcal
A_i = \{\bar r(i) = \chi_i,\,\dotsc,\,\bar r(k_0) = \chi_{k_0}\}$ and $\mathcal
B_i = \{ \bar r(i) = \chi_i \} $. Under the assumption that $\mathcal A_k$ is
not the null set, we need to evaluate the conditional probability
\begin{equation}\label{eq:P(r(k+1)=...}
\Pr(\bar r(k+1) = \chi_{k+1} \mid \mathcal A_k)
=
{\Pr(\mathcal A_{k+1})}/{\Pr(\mathcal A_k)}. 
\end{equation}
Remark that, since $\mathcal A_i \neq \emptyset$, we have 
\begin{equation}\label{eq:sigma(k)=gamma_k:merged}
\sigma(k) = \gamma_{\,k}
,\quad 
\rho(k)= \delta_k. 
\end{equation}

First, assume that $f(\beta_{k+1}) = 1$. Then, by
Definition~\ref{Def:HiddenMarkov}, an observation occurs at time~$k+1$, i.e., we
have $\tau(k+1) = k+1$ and $\sigma(k+1) = r(k+1)$. This implies that
\begin{equation*}
\mathcal B_{k+1}
=
\{ r(k+1) = \alpha_{k+1}, s(k+1) = \beta_{k+1}, 
\alpha_{k+1} = \gamma_{\,k+1}, 1 = \delta_{k+1}\}. 
\end{equation*}
Therefore, since $\mathcal A_{k+1} = \mathcal A_k \cap \mathcal B_{k+1}$,
\begin{equation}\label{eq:calEk+1:first}
\mathcal A_{k+1}
=
\{\alpha_{k+1} = \gamma_{\,k+1}, \delta_{k+1} = 1\}
\cap \{r(k+1) = \alpha_{k+1}, s(k+1) = \beta_{k+1}\} \cap \mathcal A_k
\end{equation}
and hence 
\begin{equation}\label{eq:for12.1}
\Pr(\mathcal A_{k+1})
=
\onev(\alpha_{k+1} = \gamma_{\,k+1}, \delta_{k+1} = 1)
\Pr(\{r(k+1) = \alpha_{k+1}, s(k+1) = \beta_{k+1}\} \cap \mathcal A_k).
\end{equation}
The probability appearing in the last term of this equation can be computed as
\begin{equation}\label{eq:for12.2}
\begin{aligned}
&
\Pr(\{r(k+1) = \alpha_{k+1}, s(k+1) = \beta_{k+1}\} \cap \mathcal A_k)
\\
=&
\Pr(\mathcal A_k) \Pr(r(k+1) = \alpha_{k+1}, s(k+1) = \beta_{k+1} \mid \mathcal A_k) 
\\
=&
\Pr(\mathcal A_k) \Pr(r(k+1) = \alpha_{k+1}, s(k+1) = \beta_{k+1} 
\mid r(k)=\alpha_k, s(k) = \beta_k)
\\
=&
\Pr(\mathcal A_k)  {p_{\alpha_k \alpha_{k+1}} q_{\beta_k \beta_{k+1}}}, 
\end{aligned}
\end{equation}
where we have used the fact that both $r$ and $s$ are time-homogeneous Markov
chains. Thus, from equations \eqref{eq:P(r(k+1)=...}, \eqref{eq:for12.1}, and
\eqref{eq:for12.2}, we conclude that {for the case of $f(\beta_{k+1}) = 1$,}
\begin{equation}\label{eq:hidden1}
\Pr(\bar r(k+1) = \chi_{k+1} \mid \mathcal A_k) 
=
\onev(\alpha_{k+1} = \gamma_{\,k+1}, \delta_{k+1} = 1) p_{\alpha_k \alpha_{k+1}} q_{\beta_k \beta_{k+1}}.
\end{equation}

Second, consider the case where $f(\beta_{k+1}) = 0$. In this case, the Markov
mode $r$ is not observed at time $k+1$, hence, we have $\tau(k+1) = \tau(k)$
and $\sigma(k+1) = \sigma(k)$. Therefore, using equations
\eqref{eq:sigma(k)=gamma_k:merged}, in the same way as we derived
\eqref{eq:calEk+1:first}, we can show that
\begin{equation}\label{eq:calEk+1:second}
\mathcal A_{k+1}
=
\{\gamma_{\,k+1} = \gamma_{\,k},  \lfloor \delta_{k+1} - \delta_k -1 \rfloor_T = 0\} 
\cap \{r(k+1) = \alpha_{k+1}, s(k+1) = \beta_{k+1}\} \cap \mathcal A_k
\end{equation}
and hence 
\begin{equation*}
\Pr(\mathcal A_{k+1})
=
\onev(\gamma_{\,k+1} = \gamma_{\,k},  \lfloor \delta_{k+1} \!- \!\delta_k -1 \rfloor_T = 0) 
\Pr(\{r(k+1) = \alpha_{k+1}, s(k+1) = \beta_{k+1}\} \cap \mathcal A_k).
\end{equation*}
Therefore, from equations \eqref{eq:P(r(k+1)=...}, \eqref{eq:for12.2}, and
\eqref{eq:calEk+1:second}, we show that, if $f(\beta_{k+1}) = 0$, then
\begin{equation}\label{eq:hidden2}
\Pr(\bar r(k+1) = \chi_{k+1} \mid \mathcal A_k) 
=
\onev(\gamma_{\,k} = \gamma_{\,k+1},  \lfloor \delta_{k+1} - \delta_k-1 \rfloor_T = 0) 
p_{\alpha_k \alpha_{k+1}} q_{\beta_k \beta_{k+1}}.
\end{equation}

Since the probabilities \eqref{eq:hidden1} and \eqref{eq:hidden2} do not depend
on $k_0$, letting $k_0 = k$ and $k_0 = 0$ in \eqref{eq:hidden1} and
\eqref{eq:hidden2}, we obtain
\begin{equation*}
\Pr(\bar r(k+1) = \chi_{k+1} \mid \bar r(k) = \chi_k, \dotsc,  \bar r(k_0) = \chi_{k_0})
=
\Pr(\bar r(k+1) = \chi_{k+1} \mid \bar r(k) = \chi_k)
\end{equation*}
for every $k\geq 0$. This shows that $\bar r$ is a Markov chain since
$\chi_{k_0}, \dotsc, \chi_{k+1} \in \mathfrak X$ are arbitrary. Moreover, since
the probabilities \eqref{eq:hidden1} and \eqref{eq:hidden2} do not depend on
$k$, we conclude that the Markov chain $\bar r$ is time-homogeneous and its
transition probabilities are given by \eqref{eq:bar_p}.
\end{proof}

Lemma~\ref{lem:extendedMarkov} states that the closed-loop $\Sigma_K$ can be
represented as a Markov jump linear system with its mode being the extended
Markov chain~$\bar r$. This observation leads us to the following definitions.
For $\chi, \chi' \in \mathfrak X$, we denote the transition probabilities of the
Markov chain $\bar r$ (eq.~\eqref{eq:bar_p} in Lemma~\ref{lem:extendedMarkov})
by \mbox{$\bar p_{\chi \chi'} = \Pr(\bar r(k+1) = \chi' \mid \bar r(k) =
\chi)$}. Then, we introduce the Markov jump linear system
\begin{equation*}
\bar{\Sigma}_K:\begin{cases}
\begin{aligned}
\bar x(k+1) &= A_{K, \theta(k)} \bar x(k) + E_{K, \theta(k)} \bar w(k),
\\
\bar z(k)   &= C_{K, \theta(k)} \bar x(k), 
\end{aligned}
\end{cases}
\end{equation*}
where $\theta$ is the time-homogeneous Markov chain taking values in $\mathfrak
X$ whose transition probabilities are $\Pr(\theta(k+1) = \chi' \mid
\theta(k) = \chi) = \bar p_{\chi \chi'}$, and the matrices $A_{K, \chi}$,
$C_{K, \chi}$, and $E_{K, \chi}$ are defined by
\begin{equation}\label{eq:_K,chi_things}
A_{K, \chi} = A_\alpha + B_\alpha K_{\gamma, \delta}, 
\quad
C_{K, \chi} = C_\alpha + D_\alpha K_{\gamma, \delta}, 
\quad
E_{K, \chi} = E_\alpha, 
\end{equation}
for each $\chi = (\alpha, \beta, \gamma, \delta) \in \mathfrak{X}$. We
sometimes denote $\bar x$ and $\bar z$ by $\bar x(\cdot; \bar x_0, \theta_0,
\bar w)$ and $\bar z(\cdot; \bar x_0, \theta_0, \bar w)$ whenever we need to
clarify initial conditions and disturbances $\bar w$.

The next corollary of Lemma~\ref{lem:extendedMarkov} plays the key role in this
paper.

\begin{corollary}\label{cor:equivalence}
Assume that $x_0 = \bar x_0$, $w$ and $\bar w$ have the same probability
distribution, and $\bar r_0$ and $\theta_0$ have the same probability
distribution. Then, the stochastic processes $x(\cdot; x_0, \bar r_0, w)$ and
$\bar x(\cdot; \bar x_0, \theta_0, \bar w)$ have the same probability
distribution. Also, under the same assumption, the stochastic
processes~$z(\cdot; x_0, \bar r_0, w)$ and $\bar z(\cdot; \bar x_0, \theta_0,
\bar w)$ have the same probability distribution.
\end{corollary}

\begin{proof}
By the assumption, the Markov chains $\bar r$ and $\theta$ have the same initial
distribution. The chains also have the same transition probabilities from
Lemma~\ref{lem:extendedMarkov} and the definition of $\theta$. Therefore, $\bar
r$ and~$\theta$ have the same probability distribution. Also, we notice that, by
the definition of the matrices in \eqref{eq:_K,chi_things}, the
system~$\Sigma_K$ admits the following representation:
\begin{equation*}
\begin{aligned}
x(k+1) &= A_{K, \bar r(k)} x(k) + E_{\bar r(k)} w(k),
\\
z(k)   &= C_{K, \bar r(k)} x(k). 
\end{aligned}
\end{equation*}
Therefore, $\Sigma_K$ has the same dynamics as $\bar \Sigma_K$. In conclusion,
the claim holds true under the assumptions stated in the corollary.
\end{proof}

Using Corollary~\ref{cor:equivalence}, we can characterize the performance
measures of the closed-loop system~$\Sigma_K$. The next proposition provides a
characterization for mean square stability:

\begin{proposition}\label{prop:stbl}
For $R = \{R_\chi \}_{\chi \in \mathfrak X} \subset \mathbb{R}^{n\times n}$ and
$\chi \in \mathfrak X$, define $\bar{\mathcal D}_\chi(R) =
\sum_{\chi'\in\mathfrak X} \bar p_{\chi' \chi} R_{\chi'} \in
\mathbb{R}^{n\times n}$. Then, the following statements are equivalent:
\begin{enumerate}
\item \label{item:stbl} $\Sigma_K$ is mean square stable; 

\item \label{item:barstbl}$\bar \Sigma_K$ is mean square stable; 

\item \label{item:stblcond} There exist positive-definite matrices~$Q_\chi
\in\mathbb{R}^{n\times n}$  for every $\chi \in \mathfrak X$, such that $Q_\chi -
\bar{\mathcal D}_\chi(A_K Q A_K^\top) > 0$.
\end{enumerate}
\end{proposition}

\begin{proof}
The equivalence [\ref{item:barstbl} $\Leftrightarrow$ \ref{item:stblcond}]
immediately follows from the standard theory of Markov jump linear systems (see,
e.g.,  \cite{Costa2005}). Let us prove [\ref{item:barstbl} $\Rightarrow$
\ref{item:stbl}]. Assume that $\bar \Sigma_K$ is mean square stable. In order to
show that $\Sigma_K$ is mean square stable, let us take arbitrary $x_0 \in
\mathbb{R}^n$ and $\bar r_0 \in \mathfrak X$. Then, by
Corollary~\ref{cor:equivalence} and the mean square stability of $\bar
\Sigma_K$, we can show \mbox{$E[\norm{x(k; x_0, \bar r_0, 0)}^2] = E[\norm{\bar
x(k; x_0, \bar r_0, 0)}^2] \leq C\lambda^k\norm{x_0}^2$} for some $C>0$ and
$\lambda \in [0, 1)$, which implies mean square stability of $\Sigma_K$. We
can prove [\ref{item:stbl} $\Rightarrow$ \ref{item:barstbl}] in the same way.
\end{proof}

The following proposition characterizes the $H_2$ norm of $\Sigma_K$:

\begin{proposition}\label{prop:H2}
Let $\gamma > 0$ be arbitrary. Assume that $\bar r_0$ and $\theta_0$ follow the
same distribution $\bar \mu$. Then, the following statements are equivalent:
\begin{enumerate}
\item \label{item:2<gamma} $\Sigma_K$ is mean square stable and
$\norm{\Sigma_K}_2^2 < \gamma$;
 
\item \label{item:bar2<gamma} $\bar \Sigma_K$ is mean square stable and
$\norm{\bar \Sigma_K}_2^2 < \gamma$;

\item \label{item:2cond} There exist a family of positive-definite matrices
$Q_\chi \in \mathbb{R}^{n\times n}$  for every $\chi \in \mathfrak X$, such that
$\sum_{\chi \in \mathfrak X} \tr (C_{K, \chi} Q_\chi  C_{K, \chi}^\top) <
\gamma$ and $\bar{\mathcal D}_\chi( A_K Q A_K^\top +\bar \mu  E_K  E_K^\top)  <
Q_\chi$, where the set $\bar \mu 
E_K  E_K^\top \subset \mathbb{R}^{n\times n}$ indexed by $\mathfrak X$
is defined as $(\bar \mu  E_K E_K^\top)_\chi = \bar \mu(\chi) E_{K, \chi} 
E_{K, \chi}^\top$.
\end{enumerate}
\end{proposition}

\begin{proof}
The proof of the equivalence [\ref{item:2<gamma} $\Leftrightarrow$
\ref{item:bar2<gamma}] immediately follows from Corollary~\ref{cor:equivalence}.
Also, the equivalence~[\ref{item:bar2<gamma} $\Leftrightarrow$ \ref{item:2cond}]
is a direct consequence of a standard result~\cite[Proposition~4]{DoVal2002} in
the theory of Markov jump linear systems. The details are omitted.
\end{proof}

Finally, the following proposition characterizes the $H_\infty$ norm: 

\begin{proposition}\label{prop:Hinf}
For matrices $Z = \{Z_{\chi, \chi'}  \}_{\chi, \chi' \in \mathfrak X} \subset
\mathbb{R}^{n\times n}$, define $\bar{\mathcal F}_\chi(Z) = \sum_{\chi' \in
{\mathfrak X}} \bar p_{\chi \chi'}Z_{\chi, \chi'} \in \mathbb{R}^{n\times n}$. 
Let $\gamma > 0$ be arbitrary. Consider the following statements: 
\begin{enumerate}
\item \label{item:inf<gamma}$\Sigma_K$ is mean square stable and 
$\norm{\Sigma_K}_\infty^2 < \gamma$; 
 
\item  \label{item:barinf<gamma} $\bar \Sigma_K$ is mean square stable and 
$\norm{\bar \Sigma_K}_\infty^2 < \gamma$;

\item \label{item:infcond} There exist matrices $G_\chi \in \mathbb{R}^{n\times n}$, $H_\chi \in \mathbb{R}^{n\times
n}$, $X_\chi \in \mathbb{R}^{n\times n}$, and $Z_{\chi, \chi'} \in
\mathbb{R}^{n\times n}$ ($\chi, \chi' \in \mathfrak X$) such that
\begin{gather*}
\begin{bmatrix}
G_{\chi} + G_{\chi}^\top - X_\chi& \star & \star & \star
\\
O & \gamma I & \star & \star
\\
A_{K, \chi}  G_{\chi}& E_{K, \chi}  & H_\chi + H_\chi^\top- \bar{\mathcal F}_\chi(Z) & \star
\\
C_{K, \chi} G_{\chi} & O & O & I
\end{bmatrix}
> 0,
\quad
\begin{bmatrix}
Z_{\chi, \chi'} & \star 
\\ 
H_\chi &  X_{\chi'}
\end{bmatrix} > 0, 
\end{gather*}
for all $\chi, \chi' \in \mathfrak X$.
\end{enumerate}
Then, \ref{item:inf<gamma} and \ref{item:barinf<gamma} are equivalent. Moreover,
\ref{item:infcond} implies \ref{item:inf<gamma} and \ref{item:barinf<gamma}.
\end{proposition}

\begin{proof}
The implication [\ref{item:infcond} $\Rightarrow$ \ref{item:barinf<gamma}] is a
direct consequence of \cite[Theorem 1]{Goncalves2012}. Let us prove
[\ref{item:inf<gamma} $\Leftrightarrow$ \ref{item:barinf<gamma}]. It is
sufficient to show $\norm{\Sigma_K}_\infty = \norm{\bar \Sigma_K}_\infty$. Let
$\bar{\Theta}_k$ denote the $\sigma$-algebra generated by the random variables
$\{\theta(k), \dotsc, \theta(0)\}$. Define $\bar \ell^2$ as the space of
stochastic processes $\bar f = \{\bar f(k)\}_{k=0}^\infty$ such that
$\sum_{k=0}^\infty E[\norm{\bar f(k)}^2]$ is finite and, for each $k\geq 0$,
$\bar f(k)$ is an $\mathbb{R}^n$\nobreakdash-valued and $\bar
\Theta_k$-measurable random variable. Define the norm of~$\bar f\in\bar \ell^2$
by $\norm{\bar f}_2^2= \sum_{k=0}^\infty E[\norm{\bar f(k)}^2]$. Then, the
$H_\infty$ norm of $\bar \Sigma_K$ is given by
\begin{equation*}
\norm{\bar \Sigma_K}_\infty = \sup_{\theta_0 \in
\mathfrak X} \sup_{\bar w \in \bar \ell^2 \backslash\{0\}} \frac{\norm{\bar z(\cdot; 0, \theta_0, \bar w)}_2}{\norm{\bar
w}_2}.
\end{equation*} To show that this norm equals $\norm{\Sigma_K}_\infty$, it is
sufficient to show that, if $\bar r_0 = \theta_0$, then 
\begin{equation}\label{eq:pf:prop:Hinf:}
\sup_{w \in \ell^2 \backslash\{0\}}
\frac{\norm{z(\cdot; 0, \bar r_0, w)}_2}{\norm{w}_2} = 
\sup_{\bar w \in \bar \ell^2 \backslash\{0\}} \frac{\norm{\bar z(\cdot; 0, \theta_0, \bar w)}_2}{\norm{\bar
w}_2}.
\end{equation}
By Corollary~\ref{cor:equivalence}, the only difference between both sides of
the equality \eqref{eq:pf:prop:Hinf:} is the spaces~$\ell^2$ and  $\bar \ell^2$.
In other words, to complete the proof, it is sufficient to show that $\Theta_k$,
the $\sigma$-algebra generated by the random variables $\{r(k), s(k), \dotsc,
r(0), s(0)\}$, coincides with the one generated by  $\{\bar r(k), \dotsc, \bar
r(0)\}$. This is obvious because $\sigma(k)$ and $\rho(k)$ are images of $(r(k),
s(k), r(k-1), s(k-1), \dotsc, r(0), s(0))$ under measurable functions.
\end{proof}

\section{Design of feedback gains via linear matrix inequalities}\label{sec:design}

Based on the performance characterizations presented in the previous section, we now propose a formulation based on Linear Matrix Inequalities (LMI) to design feedback control laws for stabilization, $H_2$, and $H_\infty$ control of discrete-time Markov jump linear systems under hidden-Markovian observations of the mode signals.

The next theorem provides an LMI formulation to solve the stabilization problem stated in Problem~\ref{prb:stbl}:

\begin{theorem}\label{thm:stblization}
Assume that the matrices \mbox{$R_\chi \in \mathbb{R}^{n\times n}$},
$G_{\gamma,\delta}\in \mathbb{R}^{n\times n}$, and $F_{\gamma,\delta}\in
\mathbb{R}^{m\times n}$ ($\chi = (\alpha, \beta, \gamma, \delta) \in \mathfrak
X$) satisfy the linear matrix inequality 
\begin{equation}\label{eq:lmi:stblization}
\begin{bmatrix}
R_\chi 
& 
A_\alpha  G_{\gamma, \delta} + B_{\alpha} F_{\gamma, \delta}
\\
\star
& 
G_{\gamma, \delta} + G_{\gamma, \delta}^\top - \mathcal D_\chi(R)
\end{bmatrix} > 0
\end{equation}
for every $ \chi = (\alpha, \beta, \gamma, \delta) \in \mathfrak X$. For
each $\gamma \in [N]$ and $\delta \in \oneto T$, define
\begin{equation}\label{eq:defK}
K_{\gamma, \delta} = F_{\gamma, \delta}G_{\gamma, \delta}^{-1}. 
\end{equation}
Then, the resulting closed-loop system $\Sigma_K$ is mean square stable.
\end{theorem} 

\begin{proof}
Assume that $R_\chi \in \mathbb{R}^{n\times n}$, {$G_{\gamma,\delta}\in
\mathbb{R}^{n\times n}$}, and {$F_{\gamma,\delta}\in \mathbb{R}^{m\times n}$}
satisfy \eqref{eq:lmi:stblization}, and define $K$ by \eqref{eq:defK}. Our proof
is based on an argument proposed in \cite{DoVal2002}. Since
$\mathfrak X$ is a finite set, there exists an $\epsilon>0$ such that $Q_\chi =
\mathcal D_\chi(R) + \epsilon I$ satisfies
\begin{equation}\label{eq1:pf:thmstblzation}
\begin{bmatrix}
R_\chi 
& 
A_{K, \chi}  G_{\gamma, \delta}
\\
\star
& 
G_{\gamma, \delta} + G_{\gamma, \delta}^\top - Q_\chi
\end{bmatrix} > 0, 
\end{equation}
where we have used $A_\alpha  G_{\gamma, \delta} + B_{\alpha} F_{\gamma, \delta}
= A_{K, \chi}G_{\gamma, \delta}$. Since this inequality implies $R_\chi > 0$, we
have $Q_\chi \geq \epsilon I > 0$. We here recall that, for a positive definite
matrix $A \in \mathbb{R}^{n\times n}$ and another matrix~$B\in
\mathbb{R}^{n\times n}$, it holds that (see, e.g., \cite{DoVal2002})
$BA^{-1}B^\top \geq B+B^\top -A$. Using this fact in
\eqref{eq1:pf:thmstblzation}, we obtain
\begin{equation}\label{eq2:pf:thmstblzation}
\begin{bmatrix}
R_\chi 
& 
A_{K, \chi}  G_{\gamma, \delta}
\\
\star
& 
G_{\gamma, \delta}^\top Q_\chi^{-1} G_{\gamma, \delta}
\end{bmatrix} > 0. 
\end{equation}
Also, from \eqref{eq1:pf:thmstblzation}, we see that $G_{\gamma, \delta} +
G_{\gamma, \delta}^\top > 0$ and therefore $G_{\gamma, \delta}$ is invertible.
Hence, we can take the Schur complement of the positive-definite matrix in
\eqref{eq2:pf:thmstblzation} with respect to $R_\chi$ to obtain $R_\chi - A_{K,
\chi} Q_\chi A_{K, \chi}^\top > 0$. Applying  to this inequality the operator
$\bar {\mathcal D}_\chi$, we obtain \mbox{$Q_\chi - \bar{\mathcal D}_\chi(A_K Q
A_K^\top) > \epsilon I > 0$}. Therefore, by Proposition~\ref{prop:stbl},
$\Sigma_K$ is mean square stable.
\end{proof}

Secondly, the next theorem provides an LMI formulation to solve the $H_2$ control problem stated in Problem~\ref{prb:H2}:

\begin{theorem}\label{thm:H2}
Let $\gamma>0$ be arbitrary. Assume that $W_\chi \in
\mathbb{R}^{\ell\times\ell}$, $R_\chi\in \mathbb{R}^{n\times n}$, $F_{\gamma,
\delta} \in \mathbb{R}^{m\times n}$, $G_{\gamma, \delta} \in \mathbb{R}^{n\times
n}$, and $\nu(\gamma, \delta) \geq 0$ ($\chi = (\alpha, \beta, \gamma, \delta)
\in \mathfrak X$) satisfy the following linear matrix inequalities
\begin{gather}
\begin{bmatrix}
R_\chi - \bar \mu(\chi) E_\alpha E_\alpha^\top & A_{\alpha} G_{\gamma, \delta} + B_{\alpha} F_{\gamma, \delta}
\\
\star                                    & G_{\gamma, \delta} + G_{\gamma, \delta}^\top - \bar{\mathcal D}_\chi(R)
\end{bmatrix} > 0, \label{eq:H2lmi2}
\\
\begin{bmatrix}
W_\chi & C_{\alpha} G_{\gamma, \delta} + D_{\alpha} F_{\gamma, \delta}
\\
\star  & G_{\gamma, \delta} + G_{\gamma, \delta}^\top - \bar{\mathcal{D}}_\chi(R)
\end{bmatrix} > 0, \label{eq:H2lmi1}
\\
\sum_{\chi\in{\mathfrak X}} \tr (W_\chi) < \gamma, \label{eq:trace<gamma}
\\
\sum_{\gamma=1}^{N}\sum_{\delta=1}^{T} \nu(\gamma, \delta) = 1, \label{eq:sumProbs=1}
\end{gather}
for every $\chi = (\alpha, \beta, \gamma, \delta) \in {\mathfrak X}$. Define the
feedback matrix $K$ by \eqref{eq:defK}. Then, the closed-loop system $\Sigma_K$
is mean square stable and satisfies $\norm{\Sigma_K}_2^2 < \gamma$.
\end{theorem}

We remark that the (in)equalities in Theorem~\ref{thm:H2} are indeed linear
with respect to the design variables. The linearity with respect to the
matrix variables $W_\chi$, $R_\chi$, $F_{\gamma, \delta}$, and $G_{\gamma,
\delta}$ is obvious. The linearity with respect to $\nu$ follows from
\eqref{eq:nu.lin}. We also remark that the constraint~\eqref{eq:sumProbs=1}
makes $\nu$ a probability measure.

Let us prove Theorem~\ref{thm:H2}.

\begin{proof}[Proof of Theorem~\ref{thm:H2}]
Assume that $\gamma > 0$, $W_\chi \in \mathbb{R}^{\ell\times\ell}$, $R_\chi\in
\mathbb{R}^{n\times n}$, $F_{\gamma, \delta} \in \mathbb{R}^{m\times n}$,
$G_{\gamma, \delta} \in \mathbb{R}^{n\times n}$, and $\nu(\gamma, \delta)$
satisfy \mbox{\eqref{eq:H2lmi2}--\eqref{eq:sumProbs=1}}, and let us define $K$
by \eqref{eq:defK}. In the same way as in the proof of
Theorem~\ref{thm:stblization}, there exists an $\epsilon>0$ such that $Q_\chi =
\bar{\mathcal D}_\chi(R) + \epsilon I$ satisfies
\begin{gather}
\begin{bmatrix}
R_\chi - \bar\mu(\chi) E_{K, \chi}  E_{K, \chi}^\top &  A_{K, \chi} G_{\gamma, \delta}
\\
\star                                    & G_{\gamma, \delta}^\top Q_\chi^{-1} G_{\gamma, \delta}
\end{bmatrix} > 0, \label{eq:H2lmi2.new}
\\
\begin{bmatrix}
W_\chi &  C_{K, \chi} G_{\gamma, \delta}
\\
\star  & G_{\gamma, \delta}^\top Q_\chi^{-1} G_{\gamma, \delta}
\end{bmatrix} > 0. \label{eq:H2lmi2.new2}
\end{gather}
Applying Schur complement to the matrix in the left hand side of
\eqref{eq:H2lmi2.new}, we obtain \mbox{$R_\chi - \bar \mu(\chi)  E_{K, \chi}
E_{K, \chi}  -  A_{K, \chi} Q_\chi  A_{K, \chi}^\top > 0$}. Applying the
operator $\bar{\mathcal D}_\chi$ to this inequality yields \mbox{$\bar{\mathcal
D}_\chi( A_K Q  A_K^\top +\bar \mu  E_K  E_K^\top) < \bar{\mathcal D}_\chi(R) <
Q_\chi$}. In addition, from \eqref{eq:H2lmi2.new2} it follows that
\mbox{$W_\chi >  C_{K, \chi} Q_\chi  C_{K, \chi}^\top$}. Hence, we have
\mbox{$\sum_{\chi \in \mathfrak X} \tr (C_{K, \chi} Q_\chi  C_{K, \chi}^\top) <
\sum_{\chi \in \mathfrak X} \tr( W_\chi) < \gamma$} by \eqref{eq:trace<gamma}.
Therefore, by Proposition~\ref{prop:H2}, $\Sigma_K$ is mean square stable and
satisfies $\norm{\Sigma_K}_2^2 < \gamma$.
\end{proof}

Finally, the next theorem provides an LMI formulation to solve the $H_\infty$
control problem stated in Problem~\ref{prb:Hinfty}:

\begin{theorem}\label{thm:Hinftyt}
Assume that $\gamma>0$, $H_\chi \in \mathbb{R}^{n\times n}$, $X_\chi \in
\mathbb{R}^{n\times n}$, $Z_{\chi, \chi'} \in \mathbb{R}^{n\times n}$,
$G_{\gamma, \delta} \in \mathbb{R}^{n\times n}$, and $F_{\gamma,
\delta}\in\mathbb{R}^{m\times n}$ ($\chi = (\alpha, \beta, \gamma, \delta)$ and
$\chi'$ in $\mathfrak X$) satisfy the linear matrix inequalities
\begin{gather}
\begin{bmatrix}
G_{\gamma, \delta} + G_{\gamma, \delta}^\top - X_\chi& \star & \star & \star
\\
O & \gamma I & \star & \star
\\
A_{\alpha}  G_{\gamma, \delta} + B_{\alpha} F_{\gamma, \delta} & E_{\alpha}  & H_\chi + H_\chi^\top- \bar{\mathcal F}_\chi(Z) & \star
\\
C_{\alpha} G_{\gamma, \delta} + D_{\alpha} F_{\gamma, \delta}  & O & O & I
\end{bmatrix}
> 0, \label{eq:lmiinf1}
\\
\begin{bmatrix}
Z_{\chi, \chi'} & \star 
\\ 
H_\chi &  X_{\chi'}
\end{bmatrix} > 0, \label{eq:lmiinf2}
\end{gather}
for all $\chi = (\alpha, \beta, \gamma, \delta)$ and $\chi'$ in $\mathfrak X$.
Define the feedback matrix $K$ by \eqref{eq:defK}. Then, the closed-loop system
$\Sigma_K$ is mean square stable and satisfies $\norm{\Sigma_K}_\infty^2 <
\gamma$.
\end{theorem}

\begin{proof}
Assume that the inequalities \eqref{eq:lmiinf1} and \eqref{eq:lmiinf2} are
satisfied by the matrices $H_\chi \in \mathbb{R}^{n\times n}$, $X_\chi \in
\mathbb{R}^{n\times n}$, $Z_{\chi, \chi'} \in \mathbb{R}^{n\times n}$,
$G_{\gamma, \delta} \in \mathbb{R}^{n\times n}$, and $F_{\gamma,
\delta}\in\mathbb{R}^{m\times n}$. Define $K$ by \eqref{eq:defK}. Then,
\eqref{eq:lmiinf1} implies
\begin{equation*}
\begin{bmatrix}
G_{\gamma, \delta} + G_{\gamma, \delta}^\top - X_\chi& \star & \star & \star
\\
O & \gamma I & \star & \star
\\
A_{K, \chi}  G_{\gamma, \delta}  & E_{K, \chi}  & H_\chi + H_\chi^\top- \bar{\mathcal F}_\chi(Z) & \star
\\
C_{K, \chi} G_{\gamma, \delta}   & O & O & I
\end{bmatrix}
> 0.
\end{equation*}
By this inequality and \eqref{eq:lmiinf2}, Proposition~\ref{prop:Hinf}
immediately shows that $\Sigma_K$ is mean square stable and satisfies
$\norm{\Sigma_K}_\infty^2 < \gamma$, as desired.  
\end{proof}

\section{Numerical examples}\label{sec:example}

The objective of this section is to illustrate Theorems~\ref{thm:H2} and
\ref{thm:Hinftyt} by numerical examples. We will also demonstrate how the
periodicity of the feedback gain can be used to improve the performance of the closed-loop
system~$\Sigma_K$.

\begin{example}
In this example, we consider the Markov jump linear system studied in
\cite[Example~1]{Costa2014b}. The system has two modes and its
parameters are given by
\begin{equation*}
\begin{aligned}
{\setlength{\arraycolsep}{2.5pt} A_1 = \begin{bmatrix}
0.7017	&	-1.227	&	0.3931	&	-0.6368	\\
-0.4876 &	-0.6699	&	-1.7073	&	-1.0026	\\
1.8625 	&	1.3409	&	0.2279 	&	-0.1856	\\
1.1069 	&	0.3881	&	0.6856	&	-1.0540
\end{bmatrix}},\ 
&{\setlength{\arraycolsep}{3.2pt} A_2 = \begin{bmatrix}
-0.0715	&	-0.5420	&	0.6716	&	0.6250	\\
0.2792	&	1.6342	&	-0.5081	&	-1.0473	\\
1.3733	&	0.8252	&	0.8564	&	1.5357	\\
0.1798	&	0.2308	&	0.2685	&	0.4344
\end{bmatrix}}, \\
B_1 = B_2 = \begin{bmatrix}
I_2\\O_{2\times 2}
\end{bmatrix},\ 
C_1 = C_2 = \begin{bmatrix}
I_4\\O_{2\times 4}
\end{bmatrix},\ 
&D_1 = D_2 = \begin{bmatrix}
O_{4\times 2}\\ I_2
\end{bmatrix},\ 
E_1 = E_2 = I_4, 
\\
P = \begin{bmatrix}
0.6942& 0.3058
\\
0.6942& 0.3058
\end{bmatrix},\ 
&\mu_r = 
\begin{bmatrix}
0.6942& 0.3058
\end{bmatrix}, 
\end{aligned}
\end{equation*} 
where $I_n$ and $O_{n\times m}$ denote the $n\times n$ identity matrix and the
$n\times m$ zero matrix, respectively. Notice that the mode signal~$r$ is a
sequence of independently and identically distributed random variables.

\begin{figure}[tb]
\centering
\includegraphics[width=10.3cm]{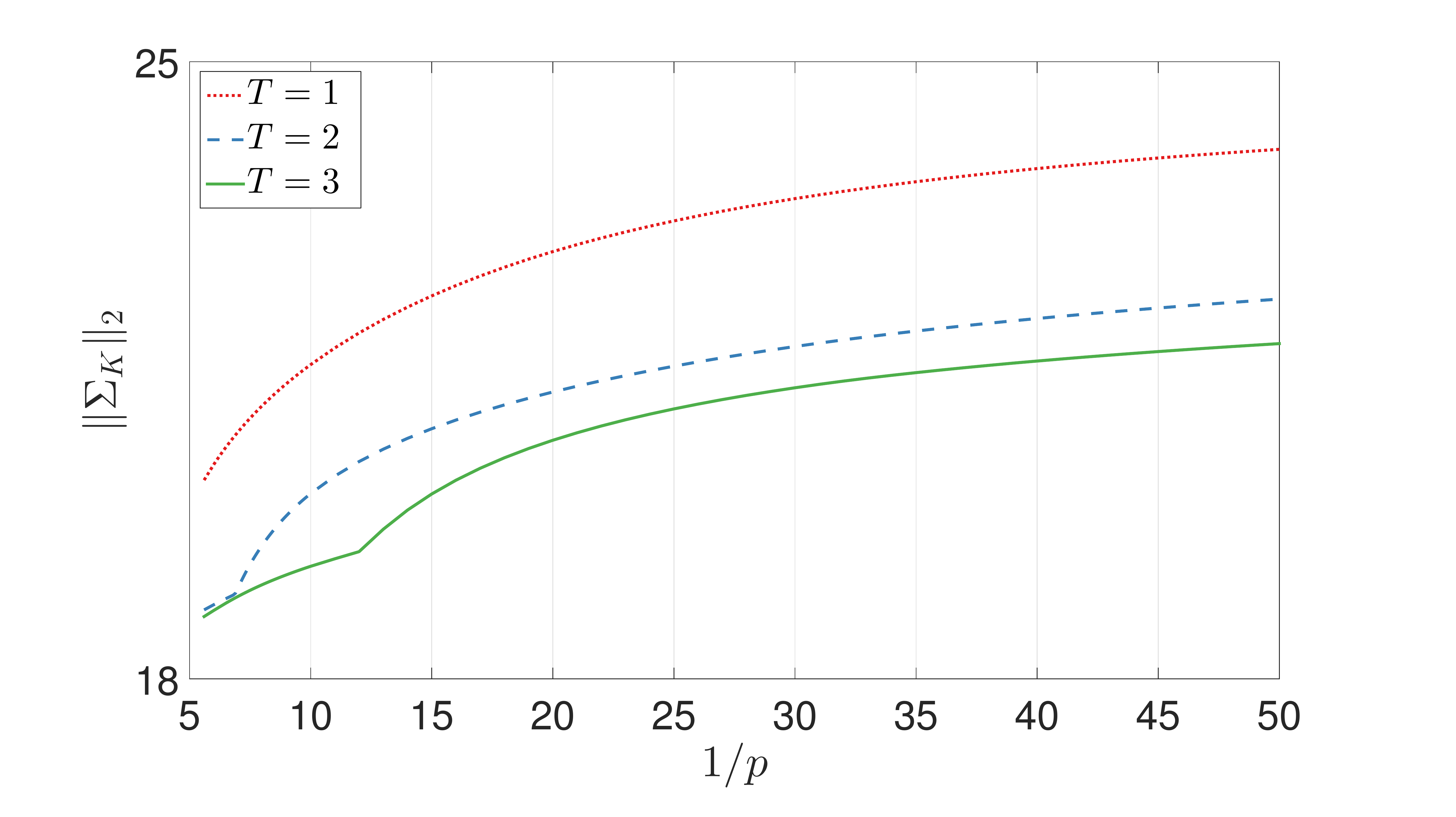}
\caption{The $H_2$ norm of the closed-loop system $\Sigma_K$ versus the  expected duration $1/p$}
\label{fig:H2}
\end{figure}

We assume that the controller observes the mode through a Gilbert-Elliot
channel, described in Example~\ref{ex:GEchannel}. For simplicity in our
presentation, we let $p=q$. Notice that, whatever value $p$ takes, the limiting
distribution of $s$ is the uniform distribution on the set $\{1, 2\}$; in other
words, the asymptotic frequency of the controller observing the mode signal $r$
is $1/2$. In addition, the expected duration of the chain $s$ staying at either
Good or Bad state depends on $p$, and is equal to $1/p$. We assume that the
initial distribution $\mu_s$ of $s$ is the uniform distribution.

We can use Theorem~\ref{thm:H2} to design stabilizing feedback gains and the
initial distribution~$\nu$ in order to achieve a small $H_2$ norm of the
closed-loop system~$\Sigma_K$ by solving the following optimization problem:
\begin{equation*}
\begin{aligned}
\minimize_{W_\chi,\,R_\chi,\,F_{\gamma, \delta},\,G_{\gamma, \delta}, \nu(\gamma,\delta)}\quad  &\gamma, 
\\
\subjectto\ \ \ \ \ \ \ \ \ \ \ \ \ &
\text{\eqref{eq:H2lmi2}--\eqref{eq:sumProbs=1}}.
\end{aligned}
\end{equation*}
Figure~\ref{fig:H2} shows the $H_2$~norms of the optimized closed-loop systems.
As expected, the larger the value of the period $T$, the smaller the attained
$H_2$ norm. We can also see that the $H_2$ norm of the closed-loop system
increases as the expected duration $1/p$ increases, although the stationary
distribution of~$s$ does not depend on $p$. We remark that this feature arising
from the Markov property of $s$ cannot be captured by the framework in
\cite{Costa2014b}, where mode observations at different time instants are
assumed to be independent events with identical probabilities.
\end{example}

\begin{example}
Consider the Markov jump linear system $\Sigma$ with the following parameters: 
\begin{equation*}
\begin{aligned}
A_1 = \begin{bmatrix}
-0.6 & -0.4
\\
-0.6 & -0.4
\end{bmatrix},\ 
A_2 = \begin{bmatrix}
-0.8& 0.4
\\
0.8 & 0.2
\end{bmatrix},\ 
&B_1 = \begin{bmatrix}
-0.3\\-0.2
\end{bmatrix},\ 
B_2 = \begin{bmatrix}
-0.2\\-0.3
\end{bmatrix},
\\
C_1 = \begin{bmatrix}
0.4&0.2
\end{bmatrix},\ 
C_2 = \begin{bmatrix}
0.1 &0.5
\end{bmatrix},\ 
&D_1 = 0.1,\ 
D_2 = -0.3,
\\
E_1 = \begin{bmatrix}
-0.3\\-0.3
\end{bmatrix},\ 
E_2 = \begin{bmatrix}
-0.2 \\-0.1
\end{bmatrix},\ 
&P = \begin{bmatrix}
0.1& 0.9
\\
0.7& 0.3
\end{bmatrix}.
\end{aligned}
\end{equation*}
From the standard theory of Markov jump linear systems~\cite{Costa2005}, one can
check that $\Sigma$ is not mean square stable when $u \equiv 0$. We use the
observation process with independent and identically distributed failures (described in Example~\ref{ex:IID}).
\begin{figure}[tb]
\centering
\includegraphics[width=10.3cm]{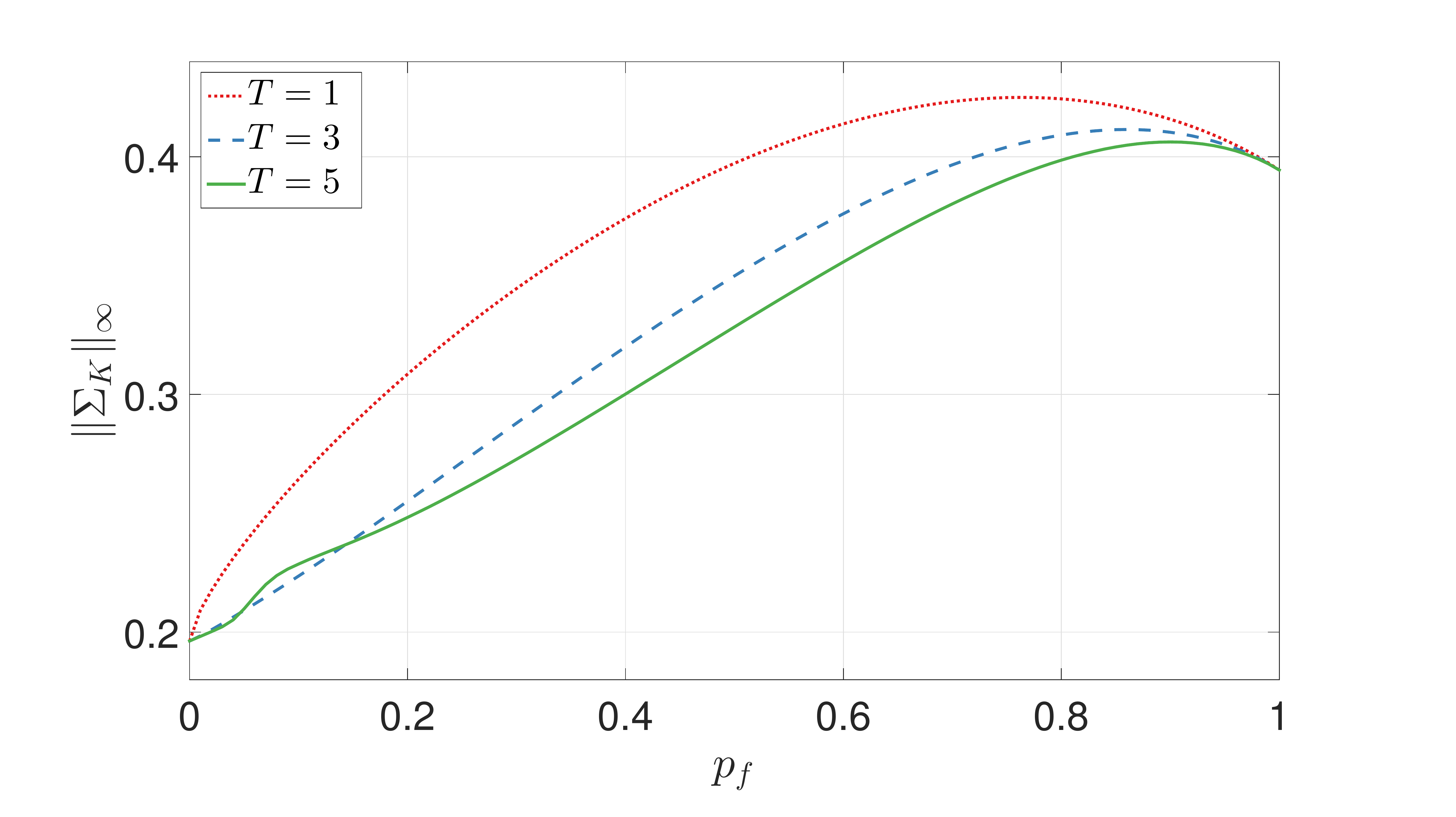}
\caption{The $H_\infty$ norm of the closed-loop system  $\Sigma_K$ versus the failure probability $p_f$}
\label{fig:Hinf}
\vspace{.75cm}
\includegraphics[width=10.3cm]{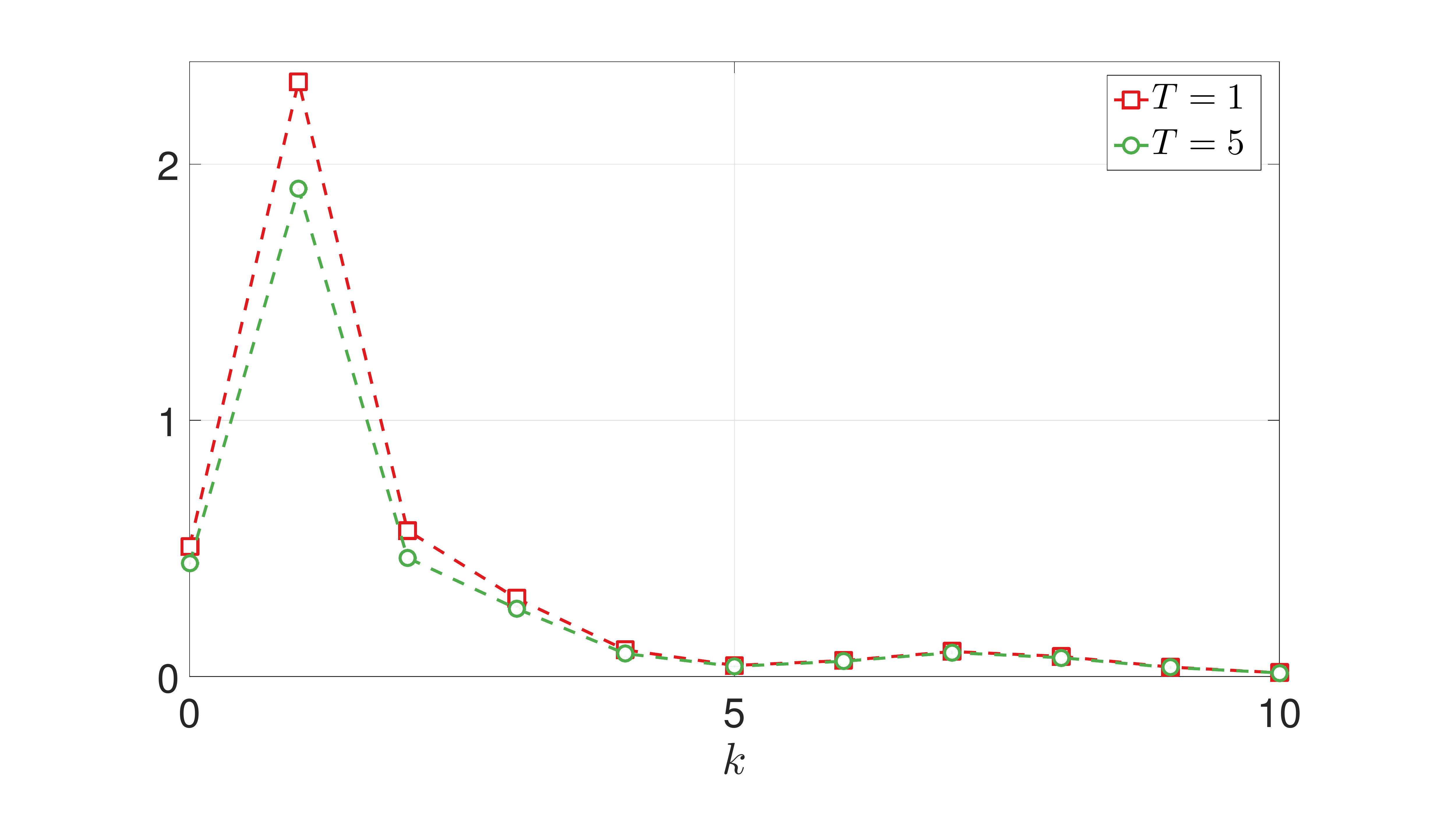}
\caption{Sample averages of $\norm{z(k)}^2$}
\label{fig:HinfExample}
\end{figure}%
In order to design stabilizing feedback gains achieving a small $H_\infty$ norm
of $\Sigma_K$, we solve the following optimization problem based on
Theorem~\ref{thm:Hinftyt}:
\begin{equation*}
\begin{aligned}
\minimize_{\substack{
H_\chi,\,X_\chi,\,Z_{\chi \chi'},\,G_{\gamma,\delta},\,F_{\gamma, \delta}
\\
}}\quad  &\gamma, 
\\
\subjectto\ \ \ \ \ \ \ \ \ \ \ \,\,&
\text{(\ref{eq:lmiinf1}) and (\ref{eq:lmiinf2})}.
\end{aligned}
\end{equation*}
Figure~\ref{fig:Hinf} shows the $H_\infty$ norms of the resulting closed-loop
systems for various values of period~$T$ and failure probability~$p_f$. As we
increase the period $T$, the $H_\infty$ norm tends to decrease. However, notice
that when $p_f$ is around $0.1$, the $H_\infty$~norm attained by the controllers
with $T=3$ are better than those with $T=5$. This phenomenon can happen because
$5$ is not an integer-multiple of $3$. Remark that, for $T=6$ instead of $T=5$,
such a phenomenon will not happen because the $H_\infty$~performance obtained by
the feedback gains~$K^{(3)} = \{K^{(3)}_{\gamma, \delta} \}_{\gamma \in [2],
\delta \in [3]}$ with period $T=3$ is attained by the feedback gains $K^{(6)}$
with period~$T=6$ given as \mbox{$K^{(6)}_{\gamma, \ell} = K^{(6)}_{\gamma,
\ell+3} = K^{(3)}_{\gamma, \ell}$} for every $\ell = 1, 2, 3$. We also remark
that, by the same reason, the performance of the resulting closed-loop system
for any integer $T$ is never worse than the performance for $T=1$ (since $1$
divides $T$).

Finally, Figure~\ref{fig:HinfExample} shows the sample averages of
$\norm{z(k)}^2$ of the closed-loop systems for $T=1$ and $T=5$. For the
computation of the sample averages, we fix $x_0 = [1 \quad 2]^\top$, $r_0 = 1$,
$s_0 = 1$, and $p_f = 0.5$. The disturbance signal is chosen as $w(k) = 2
\cos(k/2)$. We generate $300$ sample paths of $r$ and $s$. Using the sample
paths, we then generate $300$ sample paths of $z$ for $T=1$ and $T=5$,
respectively. We can see that the closed-loop system with $T=5$
attenuates the disturbance signal better than that with $T=1$.
\end{example}

\section{Conclusion}

In this paper, we have studied state-feedback control of Markov jump linear
systems with hidden-Markovian observations of the mode signals. This observation
model generalizes various relevant cases previously studied in the literature on
Markov jump linear systems, such as the cases with perfect information, no
information and cluster observations of the mode signal. We have then developed
an optimization framework, based on Linear Matrix Inequalities, to design
feedback gains for stabilization, $H_2$ and $H_\infty$~control problems.
Finally, we have illustrated the effectiveness of this optimization framework
with several numerical examples.


\end{document}